\documentclass[graybox]{svmult}

\usepackage{mathbbol,amsmath}
\usepackage{mathptmx}       % selects Times Roman as basic font
\usepackage{helvet}         % selects Helvetica as sans-serif font
\usepackage{courier}        % selects Courier as typewriter font
\usepackage{type1cm}        % activate if the above 3 fonts are
\usepackage{makeidx}         % allows index generation
\usepackage{graphicx}        % standard LaTeX graphics tool
\usepackage{multicol}        % used for the two-column index
\usepackage[bottom]{footmisc}% places footnotes at page bottom
\usepackage{amsfonts}
\makeindex             % used for the subject index
\begin{document}

\title*{Generalized Derivations of  $n$-BiHom-Lie algebras}
\author{Amine  Ben Abdeljelil, Mohamed  Elhamdadi, Ivan  Kaygorodov and Abdenacer  Makhlouf}
 \authorrunning{A.  Ben Abdeljelil, M.  Elhamdadi, I.  Kaygorodov and A.  Makhlouf} 
 %for an abbreviated version of
\institute{Amine  Ben Abdeljelil \at Department of Mathematics, 
	University of South Florida, Tampa, FL 33620, U.S.A., \email{amine@mail.usf.edu}
	\and Mohamed  Elhamdadi \at Department of Mathematics, 
University of South Florida, Tampa, FL 33620, U.S.A., \email{emohamed@math.usf.edu}
\and Ivan  Kaygorodov \at Universidade Federal do ABC, CMCC, Santo Andr\'e, Brazil \email{kaygorodov.ivan@gmail.com}
\and Abdenacer  Makhlouf \at Universit\'e de Haute Alsace, IRIMAS- d\'epartement de Math\'ematiques, Mulhouse, 
France \email{Abdenacer.Makhlouf@uha.fr}}

\maketitle

\abstract{We investigate generalized derivations of  $n$-BiHom-Lie algebras.  We introduce and study properties of derivations, $(  \alpha^{s},\beta^{r}) $-derivations and generalized derivations. We also study quasiderivations of $n$-BiHom-Lie algebras.  Generalized derivations of $(n+1)$-BiHom-Lie algebras induced by $ n $-BiHom-Lie algebras are also considered.}

\section{Introduction}
\label{sec:1}
Higher $n$-ary operations and particularly ternary operations are natural generalizations of binary operations.  They appeared in many areas of mathematics and physics.  
The first ternary algebraic structure given in an axiomatic form appeared in 1949 in the work of  N. Jacobson \cite{Jack}.  He considered a Lie bracket $[x,y]$ in a Lie algebra $\mathcal{L}$ and a subspace that is closed with respect to $[[x,y],z]$ which he called a Lie triple system. In theoretical physics, the generalization of Hamiltonian mechanics by Nambu generated lead to  ternary algebras, which are generalizations of Lie algebras.  The algebraic formulation of this structure was achieved by Fillipov \cite{Filippov:nLie} and Takhtajan \cite{takhtajan,takhtajan2} based on some generalization of the Jacobi identity.

Hom-type algebras appeared also in Physics literature when studying $q$-deforma-tions of algebras of vector fields like Witt and Virasoro algebras. It turns out that usual Jacobi identity is twisted by a homomorphism. This type of algebras was studied  in Lie case first in \cite{hom06}, then extended to  associative algebras and various other non-associative algebras in \cite{MS}. $n$-ary Hom-type generalization of $n$-ary algebras were introduced in \cite{AMS}. 
Derivations and generalized derivations of many varieties of algebras and Hom-algebras were investigated in \cite{BEM, kay1, kay_lie, kay_lie2, kay12mz, kay11aa, kay14sp, kay14mz, Kokh, KP, KP16,  KP16com, KN03, LL00, bkp, ZZ10, zhel, ZCM16,zus10}.

In \cite{CG}, the authors studied  Hom-algebras from a category theoretical point of view. This approach was generalized in  \cite{Bihom1} and lead to the concepts of  BiHom-algebras where the defining identities are twisted by two morphisms instead of only one for Hom-algebras. BiHom-type $n$-ary algebras were introduced in \cite{KMS}. 

In this article, we aim to discuss generalized derivations of  $n$-BiHom-Lie algebras. The following is the organization of the paper.  Section~\ref{Prelimin} deals with the preliminary background including the main definitions.  In Section~\ref{Deriv} we introduce and study properties of derivations, $(  \alpha^{s},\beta^{r}) $-derivations and generalized derivations. Section~\ref{Quasi-Derivations} is dedicated to quasiderivations of $n$-BiHom-Lie algebras.  In section~\ref{induced} we study generalized derivations of $(n+1)$-BiHom-Lie algebras induced by $ n $-BiHom-Lie algebras.  

\section{Basic review of $n$-BiHom-Lie algebras}
\label{Prelimin}

\begin{definition}
	A quadruple  $ (\mathfrak{g}, [\cdot,\cdot,\cdot], \alpha, \beta ) $, where $ \mathfrak{g} $ is a vector space, $ \alpha, \beta $ are linear maps of $ \mathfrak{g} $, and $ [\cdot,\cdot,\cdot] : \mathfrak{g}^{\otimes3} \rightarrow \mathfrak{g}$ is a $ 3 $-linear map, is called a $ 3 $-BiHom-Lie algebra if the following conditions are satisfied,
	\begin{enumerate}
		\item $ \alpha \circ \beta=\beta \circ \alpha $.
		%	\item $ \alpha, \beta $ are a mprphisms of algebras: For any $ x,y,z \in \mathfrak{g} $, $ \alpha([x,y,z])=[\alpha(x),\alpha(y),\alpha(z)] $. (same for $ \beta $).
		\item $ [\beta(x_{1}),\beta(x_{2}),\alpha(x_{3})]=Sgn(\sigma)[\beta(x_{\sigma(1)}),\beta(x_{\sigma(2)}),\alpha(x_{\sigma(3)})] $, for all $ x_{1},x_{2},x_{3} \in \mathfrak{g} $ and $ \sigma \in S_{3} $.
		\item $ [\beta^{2}(x_{1}),\beta^{2}(x_{2}),[\beta(y_{1}),\beta(y_{2}),\alpha(y_{3})]]=  [\beta^{2}(y_{2}),\beta^{2}(y_{3}),[\beta(x_{1}),\beta(x_{2}),\alpha(y_{1})]] \\ - [\beta^{2}(y_{1}),\beta^{2}(y_{3}),[\beta(x_{1}),\beta(x_{2}),\alpha(y_{2})]] + [\beta^{2}(y_{1}),\beta^{2}(y_{2}),[\beta(x_{1}),\beta(x_{2}),\alpha(y_{3})]]$, \\
		for all  $ x_{1},x_{2},y_{1},y_{2},y_{3} \in \mathfrak{g} $.
	\end{enumerate}
\end{definition}

\begin{definition}
	An $ n $-BiHom-Lie algebra is a vector space $ V $ equipped with an $ n $-linear map $ [\cdot, \ldots,\cdot] $ and two linear maps $ \alpha $ and $ \beta $ such that :
	\begin{enumerate}
		\item $ \alpha \circ \beta=\beta \circ \alpha $.
		%		\item $ \alpha $ and $ \beta $ are morphisms of $ n $-algebras.
		\item $[\beta(x_{1}),\ldots,\beta(x_{n-1}),\alpha(x_{n})]=Sgn(\sigma)[\beta(x_{\sigma(1)}),\ldots,\beta(x_{\sigma(n-1)}),\alpha(x_{\sigma(n)})] $, for any $ \sigma \in S_{n} $.
		\item 
		\begin{align*}
			& [\beta^{2}(x_{1}),\ldots,\beta^{2}(x_{n-1}),[\beta(y_{1}),\ldots,\beta(y_{n-1}),\alpha(y_{n})]]= \\ &  \sum_{k=1}^{n}  (-1)^{n-k} [\beta^{2}(y_{1}),\ldots,\widehat{\beta^{2}(y_{k})},\ldots,\beta^{2}(y_{n}),[\beta(x_{1}),\ldots, \beta(x_{n-1}),\alpha(y_{k})]] ,
		\end{align*}  for all $ x_{1},\ldots,x_{n-1},y_{1},\ldots,y_{n} \in V$
	\end{enumerate}
\end{definition}
We say that $ (\mathfrak{g}, [\cdot,\ldots,\cdot], \alpha, \beta ) $ is a multiplicative $ n $-BiHom-Lie algebra if $ \alpha $ and $ \beta $ are algebra morphisms and regular if they are automorphisms.

\bigskip
$ n $-BiHom-Lie algebras may be induced from $ n $-Lie algebras using two algebra morphisms as stated in the following proposition given in \cite{KMS}.
\begin{proposition}
	Let $ (V,[\cdot,\ldots,\cdot]) $ be an $ n $-Lie algebra and $ \alpha, \beta $ two morphisms of $ V $ that commute with each other. For $ x_{1},\ldots,x_{n} \in V$ define 
	\begin{equation*}
		[x_{1},\ldots,x_{n}]_{\alpha\beta}= [\alpha(x_{1}),\ldots,\alpha(x_{n-1}),\beta(x_{n})],
	\end{equation*}
	then $ (V,[\cdot,\ldots,\cdot]_{\alpha\beta},\alpha,\beta) $ is an $ n $-BiHom-Lie algebra.
\end{proposition}
In this paper we are interested in  the derivations of this particular type of $ n $-BiHom-Lie algebras, we compare them to the derivations of the original Lie algebras, and study their inherited properties.

%\begin{definition}
%	A Leibniz $ n $-BiHom-Lie algebra is a vector space $ V $ equipped with an $ n $-linear map $ [\cdot,\ldots,\cdot] $ and two linear maps $ \alpha $ and $ \beta $ such that :
%	\begin{enumerate}
%		\item $ \alpha \circ \beta=\beta \circ \alpha $.
%		\item $ \alpha $ and $ \beta $ are morphisms of $ n $-algebras.
%		\item $[\beta(x_{1}),\ldots,\beta(x_{n-1}),\alpha(x_{n})]=Sgn(\sigma)[\beta(x_{\sigma(1)}),\ldots,\beta(x_{\sigma(n-1)}),\alpha(x_{\sigma(n)})] $, for all $ x_{1},\ldots,x_{n} \in V $ and $ \sigma \in S_{n} $.
%		\item 
%		\begin{align*}
%		[\beta^{2}(x_{1}),\ldots,\beta^{2}(x_{n-1}),[\beta(y_{1}),\ldots,\beta(y_{n-1}),\alpha(y_{n})]]&=& \\ \sum_{k=1}^{n}  [\beta^{2}(y_{1}),\ldots,[\beta(x_{1}),\ldots, \beta(x_{n-1}),\alpha(y_{k})],\ldots,\beta^{2}(y_{n})] .
%		\end{align*}  For all $ x_{1},\ldots,x_{n-1},y_{1},\ldots,y_{n} \in V$
%	\end{enumerate}
%\end{definition}

\begin{example}\label{example}
	Let $ V $ be a $ 4$-dimensional vector space with the basis $\{e_1,e_2,e_3,e_4\}$. Define the following brackets: $$ [e_1,e_2,e_3]=-e_4\ ; \quad [e_1,e_2,e_4]=e_3\ ; \quad [e_1,e_3,e_4]=-e_2\ ; \quad [e_2,e_3,e_4]=e_1. $$
	With this bracket, $ (V,[\cdot,\cdot,\cdot]) $ is a $ 3$-Lie algebra. Let $ \alpha $ and $ \beta $ be two linear maps of $ V $ defined by : $$ \alpha(e_{1})=-e_{2}\ ; \quad  \alpha(e_{2})=-e_{1}\ ; \quad   \alpha(e_{3})=-e_{4}\ ; \quad  \alpha(e_{4})=-e_{3} \quad \text{and}\ , \quad \beta=-\alpha.$$ 
	Let $ [x_{1},x_{2},x_{3} ]_{\alpha \beta}= [\alpha(x_{1}),\alpha(x_{2}),\beta(x_{3})] $, be a twisted bracket defined on $ V $.  Then it follows that $ (V,[\cdot,\cdot,\cdot]_{\alpha \beta},\alpha,\beta) $ is a $ 3$-BiHom-Lie algebra.
\end{example}

Recall that a subset $ \mathcal{S} \subseteq \mathfrak{g} $ is a subalgebra of $ (\mathfrak{g},[\cdot,\ldots,\cdot],\alpha,\beta) $ if $ \alpha(\mathcal{S}) \subseteq \mathcal{S}$, $ \beta(\mathcal{S}) \subseteq \mathcal{S}$ and $ [\mathcal{S},\mathcal{S},\ldots,\mathcal{S}]\subseteq \mathcal{S} $. We say that $\mathcal{S}$ is an ideal if $ \alpha(\mathcal{S}) \subseteq \mathcal{S}$, $ \beta(\mathcal{S}) \subseteq \mathcal{S}$ and $ [\mathcal{S},\mathcal{S},\ldots,\mathfrak{g}]\subseteq \mathcal{S} $.
\begin{definition}
	The \rm{center} of $ (\mathfrak{g},[\cdot,\ldots,\cdot],\alpha,\beta) $ is the set of $ u \in \mathfrak{g} $ such that \\ $ [u,x_{1},x_{2},\ldots,x_{n-1}]=0 $ for any $ x_{1},x_{2},\ldots,x_{n-1} \in \mathfrak{g} $. The center is an ideal of $ \mathfrak{g} $ which we will denote by $ Z(\mathfrak{g}) $.
\end{definition}
A more general definition of the center is the one involving the two morphisms $ \alpha $ and $ \beta $ and we will call it the $ (\alpha,\beta) $-center.
\begin{definition}
	The $ (\alpha,\beta) $-center of $ (\mathfrak{g},[\cdot,\ldots,\cdot],\alpha,\beta) $ is the set 
	$$ Z_{(\alpha,\beta)}(\mathfrak{g})=\{ u\in \mathfrak{g}, [u,\alpha\beta(x_{1}),\ldots,\alpha\beta(x_{n-1})]=0,\ for\  any\  x_{1},\ldots,x_{n-1} \in \mathfrak{g} \} $$
\end{definition}

\begin{example}
	A direct computation gives that the $(\alpha,\beta) $-center of the $3$-BiHom-Lie algebra given in Example~\ref{example} is trivial, that is  $ Z_{(\alpha,\beta)}(\mathfrak{g})=\{0\}$.  
	%Let $ V $ be a $ 4$-dimensional vector space with the basis $\{e_1,e_2,e_3,e_4\}$. Define the following brackets: $$ [e_1,e_2,e_3]=-e_4\ ; \quad [e_1,e_2,e_4]=e_3\ ; \quad [e_1,e_3,e_4]=-e_2\ ; \quad [e_2,e_3,e_4]=e_1. $$
	%With this bracket, $ (V,[\cdot,\cdot,\cdot]) $ is a $ 3$-Lie algebra. Let $ \alpha $ and $ \beta $ be two linear maps of $ V $ defined by : $$ \alpha(e_{1})=-e_{2}\ ; \quad  \alpha(e_{2})=-e_{1}\ ; \quad   \alpha(e_{3})=-e_{4}\ ; \quad  \alpha(e_{4})=-e_{3} \quad \text{and}\ , \quad \beta=-\alpha.$$ 
	%Let $ [x_{1},x_{2},x_{3} ]_{\alpha \beta}= [\alpha(x_{1}),\alpha(x_{2}),\beta(x_{3})] $, be a twisted bracket defined on $ V $.  Then it follows that $ (V,[\cdot,\cdot,\cdot]_{\alpha \beta},\alpha,\beta) $ is a $ 3$- BiHom-Lie algebra.
\end{example}

\section{Derivations, $(  \alpha^{s},\beta^{r}) $-derivations and Generalized derivations}\label{Deriv}
\begin{definition}
	Let $ (\mathfrak{g}, [\cdot,\cdot,\cdot], \alpha, \beta ) $ be a $ 3 $-BiHom-Lie algebra. A  linear map  $D: \mathfrak{g} \rightarrow \mathfrak{g}$ is  a derivation if  for all  $x,y,z \in \mathfrak{g}$ :
	\begin{eqnarray}
	D([x,y,z])&=&[D(x),y,z]+ [x,D(y),z]+[x,y,D(z)] , 
	\end{eqnarray}
	and it is called   an $(  \alpha^{s},\beta^{r}) $-derivation of $ (\mathfrak{g}, [\cdot,\cdot,\cdot], \alpha, \beta ) $, if it satisfies :
	\begin{eqnarray}
	D \circ \alpha &=& \alpha \circ D ,  \text{ and }  D \circ \beta = \beta \circ D,\\
	D([x,y,z])&=&[D(x),\alpha^{s}\beta^{r}(y),\alpha^{s}\beta^{r}(z)]+ [\alpha^{s}\beta^{r}(x),D(y),\alpha^{s}\beta^{r}(z)]+ \nonumber\\
	&&[\alpha^{s}\beta^{r}(x),\alpha^{s}\beta^{r}(y),D(z)]\label{DRS} .  
	\end{eqnarray}
\end{definition}

%\begin{definition}
%	Let $ (\mathfrak{g}, [\cdot,\cdot,\cdot], \alpha, \beta ) $ be a $ 3 $-Leibniz-BiHom-Lie algebra and let $ D $ be a linear map of $ \mathfrak{g} $. For any positive integers $ r,s $, $ D $ is called a $ \beta^{r} \alpha^{s} $-derivation of $ (\mathfrak{g}, [\cdot,\cdot,\cdot], \alpha, \beta ) $, if :
%	\begin{itemize}
%		\item $ D \circ \alpha = \alpha \circ D $, and  $ D \circ \beta = \beta \circ D $.
%		\item $ D([x,y,z])=[D(x),\beta^{r}(y),\alpha^{s}(z)]+ [\beta^{r}(x),D(y),\alpha^{s}(z)]+[\beta^{r}(x),\alpha^{s}(y),D(z)] $, for all $ x,y,z \in \mathfrak{g} $.
%	\end{itemize}
%\end{definition}

Similarly, one can define  $(  \alpha^{s},\beta^{r}) $-derivations of   $ n $-BiHom-Lie algebras. Condition \eqref{DRS} becomes 

\begin{align*}
& D[x_{1},\ldots,x_{n} ] =
[D(x_{1}),\alpha^{s}\beta^{r}(x_{2}),\ldots,\alpha^{s}\beta^{r}(x_{n})]\\ 	& + \sum_{i=2}^{n}[\alpha^{s}\beta^{r}(x_{1}),\ldots,\alpha^{s}\beta^{r}(x_{i-1}),D(x_{i}),\alpha^{s}\beta^{r}(x_{i+1}),\ldots,\alpha^{s}\beta^{r}(x_{n})].
\end{align*}

Let   ${\rm Der}_{(  \alpha^{s},\beta^{r})}(\mathfrak{g})$ be the set of $(  \alpha^{s},\beta^{r}) $-derivations of $\mathfrak{g}$ and 
set \[{\rm Der}(\mathfrak{g}):=\bigoplus_{s \geq 0}\bigoplus_{r \geq 0} {\rm Der}_{(  \alpha^{s},\beta^{r})}(\mathfrak{g}).\]

We show that ${\rm Der}(\mathfrak{g})$  is equipped with a  Lie algebra structure. In fact, for $ D \in {\rm Der}_{(  \alpha^{s},\beta^{r})}(\mathfrak{g}) $ and $ D' \in {\rm Der}_{(  \alpha^{s'},\beta^{r'})}(\mathfrak{g}) $, we have $ [D,D'] \in {\rm Der}_{(  \alpha^{s+s'},\beta^{r+r'})}(\mathfrak{g}) $ , where $ [D,D'] $ is the standard commutator defined by $[D,D'] = DD'-D'D . $

\bigskip
A construction of the derivation extension of a multiplicative Hom-Lie algebra was given in \cite{Sheng}, in the next proposition, we generalize it to the case of BiHom-Lie algebras. 
Let $ (\mathfrak{g},[\cdot,\cdot]_{\mathfrak{g}},\alpha,\beta) $ be a BiHom-Lie algebra and let $D: \mathfrak{g} \rightarrow \mathfrak{g}$ be a linear map. Consider the vector space $ \mathcal{L}=\mathfrak{g}\oplus \mathbb{K} D $ and define the bracket $ [\cdot,\cdot]_{D} $ on $ \mathcal{L} $ by $ [u,v]_{D}=[u,v]_{\mathfrak{g}} $ and $ [u,D]_{D}=D(u) $, for any $ u,v \in \mathfrak{g} $. Also, define two morphisms of $ \mathcal{L}$ as follows: 
\[ \alpha_{D} (u,D)=(\alpha(u),D)\mbox{ and } \beta_{D} (u,D)=(\beta(u),D) \]
A direct computation gives the following proposition.
\begin{proposition}With the above definitions, 
	$ (\mathcal{L},[\cdot,\cdot]_{D},\alpha_{D},\beta_{D}) $ is a BiHom-Lie algebra if and only if $ D $ is a derivation of $ (\mathfrak{g},[\cdot,\cdot]_{\mathfrak{g}},\alpha,\beta) $.
\end{proposition}

%\begin{proof}
%Straightforward.
%\end{proof}

Let $ (V,[\cdot,\ldots,\cdot]) $ be an $ n $-Lie algebra and $ (V,[\cdot,\ldots,\cdot]_{\alpha\beta},\alpha,\beta) $ the induced $ n $-BiHom-Lie algebra where $ \alpha,\ \beta $ are the two morphisms used for this induction.   A direct computation gives the following proposition

\begin{proposition}
	Any derivation of the $ n $-Lie algebra $ (V,[\cdot,\ldots,\cdot]) $ is a derivation of  its induced $ n $-BiHom-Lie algebra $ (V,[\cdot,\ldots,\cdot]_{\alpha\beta},\alpha,\beta) $ as well.
\end{proposition}

\begin{definition}
	Let $ (\mathfrak{g},[\cdot,\ldots,\cdot],\alpha,\beta) $ be an $ n $-BiHom-Lie algebra and let $ D$ be an endomorphism of $\mathfrak{g} $. The linear map  $ D $ is called a generalized $(  \alpha^{s},\beta^{r}) $-derivation of $ \mathfrak{g} $ if there exists   $ D^{(i)} $, $ i\in \{1,\ldots,n\} $, a family of endomorphisms of $\mathfrak{g}$,  such that
	
	\begin{align*}
		& D \circ \alpha = \alpha \circ D;\  D \circ \beta = \beta \circ D \\
		& D^{(i)} \circ \alpha = \alpha \circ D^{(i)};\  D^{(i)} \circ \beta = \beta \circ D^{(i)}\  for\  any\  i,\ and \\ 
		& D^{(n)}([x_{1},\ldots,x_{n} ] )=
		[D(x_{1}),\alpha^{s}\beta^{r}(x_{2}),\ldots,\alpha^{s}\beta^{r}(x_{n})]\\ 	&+ \sum_{i=2}^{n}[\alpha^{s}\beta^{r}(x_{1}),\ldots,\alpha^{s}\beta^{r}(x_{i-1}),D^{(i-1)}(x_{i}),\alpha^{s}\beta^{r}(x_{i+1}),\ldots,\alpha^{s}\beta^{r}(x_{n})]
	\end{align*}
	for any $ x_{1},\ldots,x_{n}\in \mathfrak{g}. $
\end{definition}

The set of generalized $(  \alpha^{s},\beta^{r}) $-derivations of $ \mathfrak{g} $ is $ {\rm GDer}_{(  \alpha^{s},\beta^{r})}(\mathfrak{g}) $ and as for ${\rm  Der(\mathfrak{g})}, $  we denote  \[{\rm GDer}(\mathfrak{g}):=\bigoplus_{s \geq 0}\bigoplus_{r \geq 0} {\rm GDer}_{(  \alpha^{s},\beta^{r})}(\mathfrak{g}).\]

\begin{definition}
	An endomorphism $ D$ of   an $ n $-BiHom-Lie algebra $\mathfrak{g}$  is called a $(  \alpha^{s},\beta^{r}) $-quasiderivation if there exists an endomorphism  $ D' $ of $\mathfrak{g}$ such that 
	\begin{align*}
	& D \circ \alpha = \alpha \circ D;\  D \circ \beta = \beta \circ D \\
	& D' \circ \alpha = \alpha \circ D';\  D' \circ \beta = \beta \circ D'\  for\  every\  i,\ and \\ 
	&D'([x_{1},\ldots,x_{n}])=\sum_{i=1}^{n}[\alpha^{s}\beta^{r}(x_{1}),\ldots,\alpha^{s}\beta^{r}(x_{i-1}),D(x_{i}),\alpha^{s}\beta^{r}(x_{i+1}),\ldots,\alpha^{s}\beta^{r}(x_{n})]
	\end{align*}
	for any $ x_{1},\ldots,x_{n}\in \mathfrak{g}. $
\end{definition}
We then define 
\[{\rm QDer}(\mathfrak{g}):=\bigoplus_{s \geq 0}\bigoplus_{r \geq 0} {\rm QDer}_{(  \alpha^{s},\beta^{r})}(\mathfrak{g}).\] 

\begin{proposition}\label{prop1}
	Let $ (\mathfrak{g},[\cdot,\ldots,\cdot],\alpha,\beta) $ be a regular $n$-BiHom-Lie algebra with trivial center. Suppose that $ \mathfrak{g} = \mathcal{I} \oplus \mathcal{J} $, where $\mathcal{I}$ and  $\mathcal{J} $ are ideals of $ \mathfrak{g} $, then :
	$$ GDer(\mathfrak{g}) = GDer(\mathcal{I}) \oplus GDer(\mathcal{J}). $$
\end{proposition}

\begin{proof}
	To prove the proposition, first we will show that for any $ D \in GDer(\mathfrak{g}) $, we have $ D(\mathcal{I})\subset \mathcal{I} $ and $ D(\mathcal{J})\subset \mathcal{J} $, then it follows that the restriction of $ D $ to $ \mathcal{I} $ (resp. $ \mathcal{J} $ ) is a generalized derivation of $ \mathcal{I} $ (resp. $ \mathcal{J} $). Let $ u \in \mathcal{I} $ and let $ D(u)=a+b $, $ a\in \mathcal{I} $, $ b\in \mathcal{J} $ be the decomposition of $ D(u) $. For any $ y_{1},\ldots,y_{n-1}\in \mathfrak{g} $, we have $ [b,y_{1},\ldots,y_{n-1}] \in \mathcal{J} $. On the other hand,
	\[ [b,y_{1},\ldots,y_{n-1}]=[D(u)-a,y_{1},\ldots,y_{n-1}]= [D(u),y_{1},\ldots,y_{n-1}] -[a,y_{1},\ldots,y_{n-1}]\] since $ \mathcal{I} $ is an ideal and $ a\in \mathcal{I} $, so $ [a,y_{1},\ldots,y_{n-1}] \in \mathcal{I}  $. Moreover, for each $ 1 \leq i \leq n-1$, let $ y_{i}=\alpha^{s}\beta^{r} (x_{i})$, then:
	\begin{small}
	\begin{align*}
	& [D(u),y_{1},\ldots,y_{n-1}]=[D(u),\alpha^{s}\beta^{r} (x_{1}),\ldots,\alpha^{s}\beta^{r} (x_{n-1})] \\
	&=D^{(n)}[u,x_{1},\ldots,x_{n-1}]-\sum_{i=1}^{n-1}[\alpha^{s}\beta^{r}(u),\alpha^{s}\beta^{r}(x_{1}),\ldots,D^{(i)}(x_{i}),\alpha^{s}\beta^{r}(x_{i+1}),\ldots,\alpha^{s}\beta^{r}(x_{n-1})] .
	\end{align*}
	\end{small}
	For every $ i $, $ [\alpha^{s}\beta^{r}(u),\alpha^{s}\beta^{r}(x_{1}),\ldots,D^{(i)}(x_{i}),\alpha^{s}\beta^{r}(x_{i+1}),\ldots,\alpha^{s}\beta^{r}(x_{n-1})] \in \mathcal{I} $, \\
	so  $ \sum_{i=1}^{n-1}[\alpha^{s}\beta^{r}(u),\alpha^{s}\beta^{r}(x_{1}),\ldots,D^{(i)}(x_{i}),\alpha^{s}\beta^{r}(x_{i+1}),\ldots,\alpha^{s}\beta^{r}(x_{n-1})] \in \mathcal{I} $.\\
	Now let $ x_{i}=a_{i}+b_{i} $ be the decomposition of $ x_{i} $,
	\begin{eqnarray*}[u,x_{1},\ldots,x_{n-1}]=[u,a_{1}+b_{1},\ldots,a_{n-1}+b_{n-1}]=\\  \ [u,a_{1}+b_{1},\ldots,a_{n-1}]+[u,a_{1}+b_{1},\ldots,b_{n-1}]
	\end{eqnarray*}
	but $ [u,a_{1}+b_{1},\ldots,b_{n-1}] \in \mathcal{I} \cap \mathcal{J} =\{0\} $, so 
	\[[u,x_{1},\ldots,x_{n-1}]=[u,a_{1}+b_{1},\ldots,a_{n-2}+b_{n-2},a_{n-1}]. \]
	Similarly, $ [u,a_{1}+b_{1},\ldots,b_{n-2},a_{n-1}] =0$. Thus, 
	\[[u,x_{1},\ldots,x_{n-1}]=[u,a_{1},\ldots,a_{n-2},a_{n-1}]. \]
	Therefore, 
	\begin{align*}
	&D^{(n)}[u,x_{1},\ldots,x_{n-1}]= D^{(n)}[u,a_{1},\ldots,a_{n-1}]=[D(u),\alpha^{s}\beta^{r}(a_{1}),\ldots,\alpha^{s}\beta(a_{n-1})] \\& 
	+\sum_{i=1}^{n-1}[\alpha^{s}\beta^{r}(u),\alpha^{s}\beta^{r}(a_{1}),\ldots,D^{(i)}(a_{i}),\alpha^{s}\beta^{r}(a_{i+1}),\ldots,\alpha^{s}\beta^{r}(a_{n-1})] \in \mathcal{I}.
	\end{align*}
	Then $ [D(u),y_{1},\ldots,y_{n-1}] \in \mathcal{I} $ and so is $ [b,y_{1},\ldots,y_{n-1}] $.\\  Hence $ [b,y_{1},\ldots,y_{n-1}] \in \mathcal{I} \cap \mathcal{J}.  $ \\ÊWe conclude that $ b\in Z(\mathfrak{g})=\{0\} $ and so $ D(\mathcal{I})\subset \mathcal{I}. $
\end{proof}

\begin{remark}
	Since any derivation, quasiderivation is a generalized derivation:  \[{\rm Der}(\mathfrak{g}) \subseteq {\rm QDer(\mathfrak{g})} \subseteq {\rm GDer(\mathfrak{g})}. \] Hence  Proposition \ref{prop1} holds for ${\rm QDer(\mathfrak{g})}$ and ${\rm Der(\mathfrak{g})}$ as well, that is 
	$${\rm  Der(\mathfrak{g})} = {\rm Der}(\mathcal{I}) \oplus {\rm Der}(\mathcal{J}), $$
	and $$ {\rm QDer}(\mathfrak{g}) = {\rm QDer}(\mathcal{I}) \oplus {\rm QDer}(\mathcal{J}). $$
\end{remark}

\begin{definition}
	A linear map $ D $ is called an $(  \alpha^{s},\beta^{r}) $-central derivation of $ \mathfrak{g} $ if it satisfies 
	\[D([x_{1},\ldots,x_{n}])=[\alpha^{s}\beta^{r}(x_{1}),\ldots,\alpha^{s}\beta^{r}(x_{i-1}),D(x_{i}),\alpha^{s}\beta^{r}(x_{i+1}),\ldots,\alpha^{s}\beta^{r}(x_{n})]=0,\] 
	for all $ i \in \{1,\ldots,n\} .$
\end{definition}
The set of $(  \alpha^{s},\beta^{r}) $-central derivations is denoted by $ {\rm ZDer}_{(  \alpha^{s},\beta^{r})}(\mathfrak{g}) $ and we set 
\[{\rm ZDer}(\mathfrak{g}):=\bigoplus_{s \geq 0}\bigoplus_{r \geq 0} {\rm ZDer}_{(  \alpha^{s},\beta^{r})}(\mathfrak{g}).\] 

\begin{definition}
	The $(  \alpha^{s},\beta^{r}) $-centroid of $ (\mathfrak{g},[\cdot,\ldots,\cdot],\alpha,\beta) $ denoted by $ C_{(  \alpha^{s},\beta^{r})} (\mathfrak{g})$ is the set of linear maps $ D $ satisfying:
	\[D([x_{1},\ldots,x_{n}])=[\alpha^{s}\beta^{r}(x_{1}),\ldots,\alpha^{s}\beta^{r}(x_{i-1}),D(x_{i}),\alpha^{s}\beta^{r}(x_{i+1}),\ldots,\alpha^{s}\beta^{r}(x_{n})]\] for all $ i \in \{1,\ldots,n\} .$ We set 
	
	\[{\rm C }(\mathfrak{g}):=\bigoplus_{s \geq 0}\bigoplus_{r \geq 0} {\rm C}_{(  \alpha^{s},\beta^{r})}(\mathfrak{g}).\] 
\end{definition}
\begin{proposition}
	For any $ r,s $, we have $$  {\rm ZDer}_{(  \alpha^{s},\beta^{r})}(\mathfrak{g})= {\rm Der}_{(  \alpha^{s},\beta^{r})}(\mathfrak{g}) \cap  {\rm C}_{(  \alpha^{s},\beta^{r})} (\mathfrak{g}). $$ 
\end{proposition}
\begin{proof}
	It is clear that $ {\rm ZDer}_{(  \alpha^{s},\beta^{r})}(\mathfrak{g}) \subseteq {\rm Der}_{(  \alpha^{s},\beta^{r})}(\mathfrak{g})$ and $ {\rm ZDer}_{(  \alpha^{s},\beta^{r})}(\mathfrak{g}) \subseteq C_{(  \alpha^{s},\beta^{r})}(\mathfrak{g})$. Conversely, let $ D \in  {\rm Der}_{(  \alpha^{s},\beta^{r})}(\mathfrak{g}) \cap  {\rm C}_{(  \alpha^{s},\beta^{r})} (\mathfrak{g}) $,  so  for each $ i $ we have 
	\[D([x_{1},\ldots,x_{n}])=[\alpha^{s}\beta^{r}(x_{1}),\ldots,\alpha^{s}\beta^{r}(x_{i-1}),D(x_{i}),\alpha^{s}\beta^{r}(x_{i+1}),\ldots,\alpha^{s}\beta^{r}(x_{n})].\]
	In addition, 
	\[D([x_{1},\ldots,x_{n}])= \sum_{i=1}^{n}[\alpha^{s}\beta^{r}(x_{1}),\ldots,\alpha^{s}\beta^{r}(x_{i-1}),D(x_{i}),\alpha^{s}\beta^{r}(x_{i+1}),\ldots,\alpha^{s}\beta^{r}(x_{n})].\]
	Then $ D([x_{1},\ldots,x_{n}]) =nD([x_{1},\ldots,x_{n}])$. \\ Thus $ D([x_{1},\ldots,x_{n}])=0 $ and $ D\in {\rm  ZDer}_{(  \alpha^{s},\beta^{r})}(\mathfrak{g}) .$	 
\end{proof}

\begin{definition}
	The $(  \alpha^{s},\beta^{r}) $-quasicentroid $ {\rm QC}_{(  \alpha^{s},\beta^{r})} (\mathfrak{g})$ is the set of linear maps $ D $ such that 
	\begin{small}
	\[[D(x_{1}),\alpha^{s}\beta^{r}(x_{2})\ldots,\alpha^{s}\beta^{r}(x_{n})]=[\alpha^{s}\beta^{r}(x_{1}),\ldots,\alpha^{s}\beta^{r}(x_{i-1}),D(x_{i}),\alpha^{s}\beta^{r}(x_{i+1}),\ldots,\alpha^{s}\beta^{r}(x_{n})]\]
	\end{small} for all $ i \in \{1,\ldots,n\} .$ We set 
	\[{\rm QC}(\mathfrak{g}):=\bigoplus_{s \geq 0}\bigoplus_{r \geq 0} {\rm QC}_{(  \alpha^{s},\beta^{r})}(\mathfrak{g}).\] 
\end{definition}

\begin{lemma} \label{lem1}
	Let $ (\mathfrak{g}, [\cdot,\ldots,\cdot],\alpha,\beta) $ be an $ n $-BiHom-Lie algebra.
	
	$ (1) $ \quad $ [{\rm Der}(\mathfrak{g}),{\rm C}(\mathfrak{g})]\subseteq {\rm C}(\mathfrak{g}); $
	
	$(2)$ \quad ${\rm C}(\mathfrak{g}) \oplus {\rm Der}(\mathfrak{g})\subseteq {\rm Der}(\mathfrak{g}).$
\end{lemma}
\begin{proof} Let $ D \in Der_{(\alpha^{s},\beta^{r})}(\mathfrak{g}) $ and $ D'\in C_{(\alpha^{s'},\beta^{r'})}(\mathfrak{g}) $ for some $ s,s',r,r'. $ Let $ x_{1},\ldots,x_{n} \in \mathfrak{g}. $
	
	$ (1) $ \begin{align*}
	& [DD'(x_{1}),\alpha^{s+s'}\beta^{r+r'}(x_{2}),\ldots,\alpha^{s+s'}\beta^{r+r'}(x_{n})]\\ &= D([D'(x_{1}),\alpha^{s'}\beta^{r'}(x_{2}),\ldots,\alpha^{s'}\beta^{r'}(x_{n})])- \sum_{i=2}^{n}[\alpha^{s}\beta^{r}D'(x_{1}),\ldots,D(x_{i}),\ldots,\alpha^{s}\beta^{r}(x_{n})] \\
	&= DD'([x_{1},\ldots,x_{n}]) - \sum_{i=2}^{n}[\alpha^{s}\beta^{r}(x_{1}),\ldots,D'D(x_{i}),\ldots,\alpha^{s}\beta^{r}(x_{n})].
	\end{align*}
	On the other hand,
	\begin{align*}
	& [D'D(x_{1}),\alpha^{s+s'}\beta^{r+r'}(x_{2}),\ldots,\alpha^{s+s'}\beta^{r+r'}(x_{n})]= D'([D(x_{1}),\alpha^{s}\beta^{r}(x_{2}),\ldots,\alpha^{s}\beta^{r}(x_{n})]) \\
	& =DD'([x_{1},\ldots,x_{n}]) -D'(\sum_{i=2}^{n}[\alpha^{s}\beta^{r}(x_{1}),\ldots,D(x_{i}),\ldots,\alpha^{s}\beta^{r}(x_{n})])
	\end{align*}
	but since for each $ i $,
	
	$ D'([\alpha^{s}\beta^{r}(x_{1}),\ldots,D(x_{i}),\ldots,\alpha^{s}\beta^{r}(x_{n})]) = [\alpha^{s}\beta^{r}(x_{1}),\ldots,D'D(x_{i}),\ldots,\alpha^{s}\beta^{r}(x_{n})]$, 	
	so 
	\begin{eqnarray*} D'(\sum_{i=2}^{n}[\alpha^{s}\beta^{r}(x_{1}),\ldots,D(x_{i}),\ldots,\alpha^{s}\beta^{r}(x_{n})])=\\ \sum_{i=2}^{n}[\alpha^{s}\beta^{r}(x_{1}),\ldots,D'D(x_{i}),\ldots,\alpha^{s}\beta^{r}(x_{n})]. 
	\end{eqnarray*}
	\\
	Hence $ [[D,D'](x_{1}),\alpha^{s+s'}\beta^{r+r'}(x_{2}),\ldots,\alpha^{s+s'}\beta^{r+r'}(x_{n})] = [D,D']([x_{1},\ldots,x_{n}]).$
	
	The same proof holds for any $ i \in \{1,\ldots,n\} $. Thus $ [D,D'] \in  C_{(\alpha^{s+s'},\beta^{r+r'})}(\mathfrak{g})$.
	
	$ (2) $ Now 
	%\begin{small}
	\begin{align*}
	& D'D ([x_{1},\ldots,x_{n}])=D'([D(x_{1}),\alpha^{s}\beta^{r}(x_{2})\ldots,\alpha^{s}\beta^{r}(x_{n})]) 
	\\ & + D'(\sum_{i=2}^{n}[\alpha^{s}\beta^{r}(x_{1}),\ldots,D(x_{i}),\ldots,\alpha^{s}\beta^{r}(x_{n})])\\
	& =  [D'D(x_{1}),\alpha^{s+s'}\beta^{r+r'}(x_{2}),\ldots,\alpha^{s+s'}\beta^{r+r'}(x_{n})] 
	\\ &
	+ \sum_{i=2}^{n}[\alpha^{s+s'}\beta^{r+r'}(x_{1}),\ldots,D'D(x_{i}),\ldots,\alpha^{s+s'}\beta^{r+r'}(x_{n})].
	\end{align*}
	%\end{small}
	Then $ D'D \in Der_{(\alpha^{s+s'},\beta^{r+r'})}(\mathfrak{g}). $
\end{proof}

%\section{$ (  \alpha^{s},\beta^{r}) $-Quasi-Derivations of $ n $-BiHom-Lie Algebras}
In the following lemma, we provide some properties and relations of the subspaces of $ Der(\mathfrak{g}) $ involving in particular the subalgebra of quasiderivations $ QDer(\mathfrak{g}) $.
\begin{lemma} Let $ (\mathfrak{g}, [\cdot,\ldots,\cdot],\alpha,\beta]) $ be a multiplicative $ n $-BiHom-Lie algebra.
	
	$(1)$ \quad $[{\rm QDer}(\mathfrak{g}),{\rm QC}(\mathfrak{g})]\subseteq {\rm QC}(\mathfrak{g});$
	
	$(2)$ \quad ${\rm C}(\mathfrak{g})\subseteq {\rm QDer}(\mathfrak{g});$
	
	$(3)$ \quad $[{\rm QC}(\mathfrak{g}),{\rm QC}(\mathfrak{g})]\subseteq {\rm QDer}(\mathfrak{g});$
	
	$(4)$ \quad ${\rm QDer}(\mathfrak{g})+{\rm QC}(\mathfrak{g})\subseteq {\rm GDer}(\mathfrak{g}).$
	
	%	 $(5)$ \quad ${\rm Der}(\mathfrak{g})\oplus{\rm C}(\mathfrak{g})\subseteq {\rm QDer}(\mathfrak{g}).$
	
\end{lemma}

\begin{proof}
	$ (1) $ This inclusion is similar to $ (1) $ of Lemma \ref{lem1}.
	
	$ (2) $ It is an immediate consequence of the definition of a quasiderivation. If $ D \in  C_{(\alpha^{s},\beta^{r})}(\mathfrak{g}) $, then 
	$ \sum_{i=1}^{n}[\alpha^{s}\beta^{r}(x_{1}),\ldots,D(x_{i}),\ldots,\alpha^{s}\beta^{r}(x_{n})]=nD([x_{1},\ldots,x_{n}]) .$
	
	$ (3) $ Let $ D \in QC_{(\alpha^{s},\beta^{r})}(\mathfrak{g}) $ and $ D' \in QC_{(\alpha^{s'},\beta^{r'})}(\mathfrak{g}) $. For any $ x_{1},\ldots,x_{n} \in \mathfrak{g} $ we have 
	\begin{align*}
	& [DD'(x_{1}),\alpha^{s+s'}\beta^{r+r'}(x_{2}),\ldots,\alpha^{s+s'}\beta^{r+r'}(x_{n})]\\ &= [\alpha^{s}\beta^{r}D'(x_{1}),D\alpha^{s'}\beta^{r'}(x_{2}),\ldots,\alpha^{s+s'}\beta^{r+r'}(x_{n})] \\
	& = [\alpha^{s+s'}\beta^{r+r'}(x_{1}),D\alpha^{s'}\beta^{r'}(x_{2}),D'\alpha^{s}\beta^{r}(x_{3}),\ldots,\alpha^{s+s'}\beta^{r+r'}(x_{n})] \\ &=[D\alpha^{s'}\beta^{r'}(x_{1}),\alpha^{s+s'}\beta^{r+r'}(x_{2}),D'\alpha^{s}\beta^{r}(x_{3}),\ldots,\alpha^{s+s'}\beta^{r+r'}(x_{n})] \\
	& =  [D'D(x_{1}),\alpha^{s+s'}\beta^{r+r'}(x_{2}),\alpha^{s+s'}\beta^{r+r'}(x_{3}),\ldots,\alpha^{s+s'}\beta^{r+r'}(x_{n})].
	\end{align*}
	Then $ [[D,D'](x_{1}),\alpha^{s+s'}\beta^{r+r'}(x_{2}),\ldots,\alpha^{s+s'}\beta^{r+r'}(x_{n})]=0 $.
	
	In the same way we have $ [\alpha^{s+s'}\beta^{r+r'}(x_{1}),\ldots,[D,D'](x_{i}),\ldots,\alpha^{s+s'}\beta^{r+r'}(x_{n})] =0$ for all $ i $. Hence $ \sum_{i=1}^{n}[\alpha^{s+s'}\beta^{r+r'}(x_{1}),\ldots,[D,D'](x_{i}),\ldots,\alpha^{s+s'}\beta^{r+r'}(x_{n})]=0. $ And so $ [D,D'] \in QDer_{(\alpha^{s+s'},\beta^{r+r'})}(\mathfrak{g}) . $ 
	
	$ (4)$ Obviously.  
\end{proof}

\begin{proposition}
	If $ (\mathfrak{g},[\cdot,\ldots,\cdot],\alpha,\beta) $ is an $ n $-BiHom-Lie algebra with trivial center, then we have the following
	
	\[{\rm Der(\mathfrak{g})} \oplus {\rm C(\mathfrak{g})} \subseteq {\rm QDer(\mathfrak{g})}.\]
\end{proposition}

\begin{proof}
	Both $ {\rm Der(\mathfrak{g})}  $ and $  {\rm C(\mathfrak{g})}  $ are subspaces of $ {\rm QDer(\mathfrak{g})} $. Moreover, if $ D \in {\rm Der(\mathfrak{g})} \cap  {\rm C(\mathfrak{g})} $ then for $ u \in \mathfrak{g} $ we have 
	$ \sum_{i=1}^{n}[D(u),x_{1},\ldots,x_{n-1}]= [D(u),x_{1},\ldots,x_{n-1}] =0$, for all $ x_{1},\ldots,x_{n-1} \in \mathfrak{g} $. Therefore $ D(u) \in Z(\mathfrak{g}) $, hence $ D=0. $
\end{proof}

\section{Quasiderivations of $ n $-BiHom-Lie Algebras}\label{Quasi-Derivations}

\begin{proposition} 
	Let $ (\mathfrak{g},[\cdot,\ldots,\cdot],\alpha,\beta) $ be an $ n $-BiHom-Lie algebra over
	${\mathbb K}$ and $t$ be an indeterminate. 
	Define 
	$\breve{\mathfrak{g}}= \{\Sigma(x\otimes t+y\otimes t^{n}) \ | \ x,y\in \mathfrak{g}\} ,\ \breve{\alpha}(\breve{\mathfrak{g}})= \{\Sigma ( \alpha(x) \otimes t + \alpha (y) \otimes t^n) \ | \  x,y\in \mathfrak{g}\},$
	and $ \ \breve{\beta}(\breve{\mathfrak{g}})= \{\Sigma ( \beta(x) \otimes t + \beta (y) \otimes t^n) \ | \  x,y\in \mathfrak{g}\}$. 
	Then 
	$ (\breve{\mathfrak{g}},[\cdot,\ldots,\cdot],\breve{\alpha},\breve{\beta}) $ is a multiplicative $ n $-BiHom-Lie algebra where the bracket is given by 
	$$[x_1\otimes t^{i_1},x_2 \otimes
	t^{i_2}, \ldots ,x_n\otimes t^{i_n}]=[x_1,x_2, \ldots, x_n]\otimes t^{\sum i_j},$$
	for $ i_1, \ldots, i_n \in \{1,n\}$ . If $k > n$, we let $t^k = 0$.
	
\end{proposition}

\begin{proof}
	For any $ x,x_{1},\ldots, x_{n}\in \mathfrak{g} $ and $i, i_1, \ldots, i_n \in \{1,n\}$, we have \begin{align*}
		\breve{\alpha}\circ \breve{\beta}(x\otimes t^{i}) &= \breve{\alpha}(\beta(x)\otimes t^{i})\\
		& = \alpha \circ \beta (x) \otimes t^{i} \\
		&= \beta \circ \alpha (x) \otimes t^{i} = \breve{\beta}\circ \breve{\alpha} (x\otimes t^{i})
	\end{align*}
	then $ \breve{\alpha}\circ \breve{\beta} =  \breve{\beta}\circ \breve{\alpha}$. Also,
	\begin{align*}
		\breve{\alpha}([x_1\otimes t^{i_1},x_2 \otimes
		t^{i_2}, \ldots ,x_n\otimes t^{i_n}]) &= \breve{\alpha} ([x_1,x_2, \ldots, x_n]\otimes t^{\sum i_j}) \\
		& = [\alpha (x_1), \alpha (x_2), \ldots,\alpha  (x_n)]\otimes t^{\sum i_j} \\
		& = [\breve{\alpha} (x_1\otimes t^{i_1} ), \breve{\alpha} (x_2 \otimes t^{i_2}), \ldots,\breve{\alpha}  (x_n \otimes t^{i_n})].
	\end{align*}
	Same argument holds for $ \breve{\beta} $.
	\begin{align*}
		&[\breve{\beta}  (x_{1} \otimes t^{i_{1}}),\ldots,\breve{\beta} (x_{n-1} \otimes t^{i_{n-1}}), \breve{\alpha}(x_{n} \otimes t^{i_{n}})]=[\beta(x_{1}),\ldots, \beta(x_{n-1}),\alpha(x_{n})] \otimes t^{\sum i_j} \\
		&= Sgn(\sigma)[\beta(x_{\sigma(1)}),\ldots, \beta(x_{\sigma(n-1)}),\alpha(x_{\sigma(n)})] \otimes t^{ \sum i_j } \\
		&= Sgn(\sigma)	[\breve{\beta}  (x_{\sigma(1)} \otimes t^{i_{1}},\ldots,\breve{\beta} (x_{\sigma (n-1)} \otimes t^{i_{n-1}}), \breve{\alpha}(x_{\sigma (n)} \otimes t^{i_{n}})]
	\end{align*}
	for any $ \sigma \in S_{n}. $ Note that if $ i_{j}=n $ for some $ j $, then the bracket would be zero since in that case the sum $ \sum i_j  \mathfrak{g}eqslant n+1$ therefore $ t^{ \sum i_j } =0 $. So one may assume that $ i_1= \ldots= i_n=1 $. Finally,
	\begin{align*}
		& [\breve{\beta}^{2}  (x_{1} \otimes t^{i_{1}}),\ldots,\breve{\beta}^{2} (x_{n-1} \otimes t^{i_{n-1}}), [\breve{\beta}  (y_{1} \otimes t^{i'_{1}}),\ldots,\breve{\beta} (y_{n-1} \otimes t^{i'_{n-1}}),\breve{\alpha}(y_{n} \otimes t^{i'_{n}})]]  \\
		& =[\beta^{2}  (x_{1}) \otimes t^{i_{1}},\ldots,\beta^{2} (x_{n-1}) \otimes t^{i_{n-1}},[\beta(y_{1}),\ldots, \beta(y_{n-1}),\alpha(y_{n})] \otimes t^{\sum i'_j} ]\\
		&= [\beta^{2}  (x_{1}),\ldots,\beta^{2} (x_{n-1}) ,[\beta(y_{1}),\ldots, \beta(y_{n-1}),\alpha(y_{n})] ] \otimes t ^{\sum i_j + \sum i'_j}  \\
		&= \sum_{k=1}^{n}  (-1)^{n-k} [\beta^{2}(y_{1}),\ldots,\widehat{\beta^{2}(y_{k})},\ldots,\beta^{2}(y_{n}),[\beta(x_{1}),\ldots, \beta(x_{n-1}),\alpha(y_{k})]] \otimes t ^{\sum i_j + \sum i'_j}\\
		&= \sum_{k=1}^{n}  (-1)^{n-k} [\breve{\beta}^{2}(y_{1}\otimes t^{i'_{1}}),\ldots,\widehat{\breve{\beta}^{2}(y_{k}\otimes t^{i'_{k}})},\ldots,\breve{\beta}^{2}(y_{n}\otimes t^{i'_{n}}),[\breve{\beta}(x_{1}\otimes t^{i_{1}}),\ldots \\
		& \ldots, \breve{\beta}(x_{n-1}\otimes t^{i_{n-1}}),\breve{\alpha}(y_{k}\otimes t^{i'_{k}})]].
	\end{align*}
	Thus, $ (\breve{\mathfrak{g}},[\cdot,\ldots,\cdot],\breve{\alpha},\breve{\beta}) $ is a multiplicative $ n $-BiHom-Lie algebra.
\end{proof}

For the sake of convenience, we will write $xt ~ (xt^{n})$ instead of $x\otimes t ~ (x\otimes t^{n}).$
 If $U$ is a  subspace of $\mathfrak{g}$ such that $\mathfrak{g}=U\oplus [\mathfrak{g},\ldots,\mathfrak{g}],$ then 
$$\breve{\mathfrak{g}}=\mathfrak{g} t+\mathfrak{g} t^{n}=\mathfrak{g} t+Ut^{n}+[\mathfrak{g},\ldots,\mathfrak{g}]t^{n}.$$

Let a map $ \varphi :{\rm QDer}(\mathfrak{g})\rightarrow {\rm End}(\breve{\mathfrak{g}}) $ be defined by 
$$\varphi(D)(at+ut^{n}+bt^{n})=D(a)t+D'(b)t^{n},$$
where $D\in {\rm QDer}(\mathfrak{g}),$  $D'$ is a map related to $D$ by the definition of quasiderivation, $a\in\mathfrak{g},u\in U,b\in [\mathfrak{g},\ldots,\mathfrak{g}]$. 

\begin{proposition}
	Let $\mathfrak{g},\breve{\mathfrak{g}},\varphi$ be as above. Then
	
	$(1)$ $\varphi$ is injective and $\varphi(D)$ does not depend on the choice of $D'$;
	
	$(2)$ $\varphi({\rm QDer}(\mathfrak{g}))\subseteq {\rm Der}(\breve{\mathfrak{g}}).$
\end{proposition}

\begin{proof}
	$(1)$ If $\varphi(D_{1})=\varphi(D_{2}),$ then for all $a\in \mathfrak{g},b\in [\mathfrak{g},\ldots ,\mathfrak{g}]$ and $u\in U$ we have 
	$$\varphi(D_{1})(at+ut^{n}+bt^{n})=\varphi(D_{2})(at+ut^{n}+bt^{n}),$$
	so 
	$$D_{1}(a)t+D'_{1}(b)t^{n}=D_{2}(a)t+D'_{2}(b)t^{n},$$
	therefore $D_{1}(a)=D_{2}(a).$ Thus $D_{1}=D_{2}.$
	
	Now suppose that there exists $D''$ such that
	$$\varphi(D)(at+ut^{n}+bt^{n})=D(a)t+D''(b)t^{n},$$
	and 
	$$D''([x_{1},\ldots,x_{n}])=\sum_{i=1}^{n}[\alpha^{s}\beta^{r}(x_{1}),\ldots,\alpha^{s}\beta^{r}(x_{i-1}),D(x_{i}),\alpha^{s}\beta^{r}(x_{i+1}),\ldots,\alpha^{s}\beta^{r}(x_{n})] ,$$
	for any $x_{1},\ldots,x_{n} \in \mathfrak{g}  $, then $ D''([x_{1},\ldots,x_{n}])=D'([x_{1},\ldots,x_{n}]) $. Hence $ D''(b)=D'(b) $ and so 
	$$\varphi(D)(at+ut^{n}+bt^{n})=D(a)t+D'(b)t^{n}=D(a)t+D''(b)t^{n}.$$
	
	$ (2) $ Let $ x_{1}t^{i_{1}}, \ldots, x_{n}t^{i_{n}} \in \breve{\mathfrak{g}} $. Again, here we consider only the case when $ i_1= \ldots= i_n=1 $ since otherwise $[x_1t^{i_1}, \ldots, x_n t^{i_n}]=0.$ 
	\begin{align*}
		&\varphi(D)([x_1t, \ldots, x_nt])=\varphi(D)([x_1,\ldots, x_n]t^n) =D'([x_1, \ldots, x_n])t^n \\
		&=\sum_{i=1}^{n}[\alpha^{s}\beta^{r}(x_{1}),\ldots,\alpha^{s}\beta^{r}(x_{i-1}),D(x_{i}),\alpha^{s}\beta^{r}(x_{i+1}),\ldots,\alpha^{s}\beta^{r}(x_{n})]t^{n} \\
		&=\sum_{i=1}^{n}[\alpha^{s}\beta^{r}(x_{1})t,\ldots,\alpha^{s}\beta^{r}(x_{i-1})t,D(x_{i})t,\alpha^{s}\beta^{r}(x_{i+1})t,\ldots,\alpha^{s}\beta^{r}(x_{n})t] \\
		&=\sum_{i=1}^{n}[\breve{\alpha}^{s}\breve{\beta}^{r}(x_{1}t),\ldots,\breve{\alpha}^{s}\breve{\beta}^{r}(x_{i-1}t),\varphi (D)(x_{i}t),\breve{\alpha}^{s}\breve{\beta}^{r}(x_{i+1}t),\ldots,\breve{\alpha}^{s}\breve{\beta}^{r}(x_{n}t)].
	\end{align*}
	Hence $ \varphi(D) \in {\rm Der}_{(  \alpha^{s},\beta^{r})}(\breve{\mathfrak{g}}) $.
\end{proof}

\begin{proposition}
	Let $ \mathfrak{g} $ be a multiplicative $n$-BiHom-Lie algebra with trivial center and let $ \breve{\mathfrak{g}}, \varphi $ be as defined above. We have 
	$${\rm Der}(\breve{\mathfrak{g}})=\varphi({\rm QDer}(\mathfrak{g}))\oplus {\rm ZDer}(\breve{\mathfrak{g}}).$$
\end{proposition}

\begin{proof}
	It is obvious that $ \varphi({\rm QDer}(\mathfrak{g})) + {\rm ZDer}(\breve{\mathfrak{g}}) \subseteq {\rm Der}(\breve{\mathfrak{g}}) $ since both $ \varphi({\rm QDer}(\mathfrak{g})) $ and $ {\rm ZDer}(\breve{\mathfrak{g}}) $ are subsets of $ {\rm Der}(\breve{\mathfrak{g}}). $
	
	Moreover, since ${\rm Z}(\mathfrak{g})=\{0\}$, we have ${\rm Z}(\breve{\mathfrak{g}})=\mathfrak{g} t^n.$ Let $ g \in {\rm Der}(\breve{\mathfrak{g}}) $, so $ g(\rm Z(\breve{\mathfrak{g}})) \subseteq \rm Z(\breve{\mathfrak{g}}) $, then $ g(Ut^n)\subseteq g({\rm Z}(\breve{\mathfrak{g}}))\subseteq {\rm Z }(\breve{\mathfrak{g}})=\mathfrak{g} t^n. $ Define a map $f:\mathfrak{g} t+Ut^n+[\mathfrak{g},\ldots,\mathfrak{g}]t^n\rightarrow \mathfrak{g} t^n$ by 
	$$\ f(x)=\left\{\begin{array}{ll}g(x)\cap \mathfrak{g} t^n,& x\in \mathfrak{g} t ;\\
	g(x),& x\in Ut^n ;\\  0,& x\in [\mathfrak{g},\ldots,\mathfrak{g}]t^n.\end{array}\right.$$
	$ f $ is linear and we know that
	$$f([\breve{\mathfrak{g}}, \ldots, \breve{\mathfrak{g}}])=f([\mathfrak{g},\ldots ,\mathfrak{g}]t^n)=0,$$
	\begin{eqnarray*} &&  [\breve{\alpha}^{s}\breve{\beta}^{r}(\breve{\mathfrak{g}}), \ldots, f(\breve{\mathfrak{g}}), \ldots,  \breve{\alpha}^{s}\breve{\beta}^{r}(\breve{\mathfrak{g}})]
	\\ && Ê\subseteq [\alpha^{s}\beta^{r}(\mathfrak{g})t+\alpha^{s}\beta^{r}(\mathfrak{g})t^n,\ldots, \mathfrak{g} t^n,\ldots, \alpha^{s}\beta^{r}(\mathfrak{g})t+\alpha^{s}\beta^{r}(\mathfrak{g})t^n]=0,
	\end{eqnarray*}
	then $ f \in \rm ZDer(\breve{\mathfrak{g}}).  $ We claim that $ g-f \in  \varphi({\rm QDer}(\mathfrak{g}))  $, this implies that $ g \in  \varphi({\rm QDer}(\mathfrak{g})) + {\rm ZDer}(\breve{\mathfrak{g}}) $, hence we have equality. In fact, since 
	$$(g-f)(\mathfrak{g} t)=g(\mathfrak{g} t)-g(\mathfrak{g} t)\cap \mathfrak{g} t^{n}=g(\mathfrak{g} t)-\mathfrak{g} t^{n}\subseteq \mathfrak{g} t,~
	(g-f)(Ut^n)=0,$$ and
	$$(g-f)([\mathfrak{g}, \ldots, \mathfrak{g}]t^n)=g([\breve{\mathfrak{g}},\ldots ,\breve{\mathfrak{g}}])\subseteq
	[\breve{\mathfrak{g}},\ldots, \breve{\mathfrak{g}}]=[\mathfrak{g}, \ldots, \mathfrak{g}]t^n,$$ there exists $D,~D'\in
	{\rm End}(\mathfrak{g})$ such that for all $a\in \mathfrak{g},~b\in [\mathfrak{g}, \ldots, \mathfrak{g}]$,
	$$(g-f)(at)=D(a)t,~ (g-f)(bt^n)=D'(b)t^n.$$ 
	$ g-f \in {\rm Der}(\breve{\mathfrak{g}}) $, then 
	\begin{align*}
		\sum_{i=1}^{n}[\breve{\alpha}^{s}\breve{\beta}^{r}(a_{1}t),\ldots,\breve{\alpha}^{s}\breve{\beta}^{r}(a_{i-1}t),(g-f)(a_{i}t),\breve{\alpha}^{s}\breve{\beta}^{r}(a_{i+1}t),\ldots,\breve{\alpha}^{s}\breve{\beta}^{r}(a_{n}t)] \\ = (g-f)([a_1t, \ldots, a_nt]), 
	\end{align*}
	for all $a_1, \ldots, a_n\in \mathfrak{g}.$ Then 
	$$\sum_{i=1}^{n}[\alpha^{s}\beta^{r}(a_{1}),\ldots,D(a_{i}),\ldots,\alpha^{s}\beta^{r}(a_{n})]t^{n} = D'([a_1, \ldots, a_n])t^{n},$$
	thus 
	$$\sum_{i=1}^{n}[\alpha^{s}\beta^{r}(a_{1}),\ldots,D(a_{i}),\ldots,\alpha^{s}\beta^{r}(a_{n})] = D'([a_1, \ldots, a_n]),$$
	which means that $ D \in {\rm QDer}(\mathfrak{g}). $ Therefore, $g-f=\varphi(D)\in \varphi({\rm
		QDer}(\mathfrak{g}))$.
	
	Now if $f\in \varphi({\rm QDer}(\mathfrak{g}))\cap{\rm ZDer}(\breve{\mathfrak{g}})$, then $ f=\varphi(D) $ for some $ D \in {\rm QDer}(\mathfrak{g}) $. So 
	$$f(at+ut^n+bt^n)=\varphi(D)(at+ut^n+bt^n)=D(a)t+D'(b)t^n,$$ where $a\in \mathfrak{g},b\in [\mathfrak{g}, \ldots, \mathfrak{g}].$ Also, since $f\in {\rm ZDer}(\breve{\mathfrak{g}}),$ we have
	$$f(at+bt^n+ut^n)\in {\rm Z}(\breve{\mathfrak{g}})=\mathfrak{g} t^n.$$ That is,
	$D(a)=0,$ for all $a\in \mathfrak{g}$ and so $D=0.$ Hence $f=0.$
	
	We conclude that 
	$${\rm Der}(\breve{\mathfrak{g}})=\varphi({\rm QDer}(\mathfrak{g}))\oplus {\rm ZDer}(\breve{\mathfrak{g}}).$$
\end{proof}

\section{Generalized Derivations of $(n+1)$-BiHom-Lie Algebras induced by $ n $-BiHom-Lie algebras}\label{induced}

In \cite{n-LieAMS}, the authors investigated a construction of $(n+1)$-Hom-Lie algebras induced by $ n $-Hom-Lie algebras. The construction of $(n+1)$-BiHom-Lie Algebras induced by $ n $-BiHom-Lie algebras was studied in \cite{KMS}.  In this section, we discuss  $(\alpha^s, \beta^r)$-derivations of  $ n $-BiHom-Lie algebras that give $(\alpha^s, \beta^r)$-derivations on the induced  $(n+1)$-BiHom-Lie algebras.
%First we recall the construction of induced  $(n+1)$-Hom-Lie algebras given in  \cite{KMS}.

\begin{definition} 
	Let $A$ be a vector space, $\phi : A^n \to A$ be an $n$-linear map and $\tau$ be a linear form.
	The map $\tau$ is said to be an $(\alpha,\beta)$-twisted $\phi$-trace if it satisfies the following condition:
	\[ \forall x_1,\ldots ,x_n \in A, \tau(\phi(\beta(x_1),\ldots,\beta(x_{n-1}),\alpha(x_n))) = 0. \]
	
	We set  $\phi_\tau$  be an $(n+1)$-linear map  defined by:
	\[\forall x_1,\ldots,x_{n+1} \in A, \phi_\tau(x_1,\ldots,x_{n+1}) = \sum_{i=1}^{n+1} (-1)^{i-1} \tau(x_i)\phi(x_1,\ldots,\widehat{x_i},\ldots,x_{n+1}). \]
	% \alpha\beta trace/twisted trace
	
\end{definition}
%If the maps $\phi$, $\alpha$ and $\beta$ above are clear from the context, we will simply refer to such linear forms by twisted traces.

We recall the  construction of  an $(n+1)$-BiHom-Lie algebra using an $n$-BiHom-Lie algebra and an $(\alpha,\beta)$-twisted trace  given in \cite{KMS}:
\begin{theorem} \label{inducedbihom} %The construction theorem
	
	Let $(A,[\cdot,\ldots,\cdot],\alpha,\beta)$ be an $n$-BiHom-Lie algebra and $\tau$ an  $(\alpha,\beta)$-twisted $[\cdot,\ldots,\cdot]$-trace. If the following conditions 
	\[\forall x,y \in A, \tau(\alpha(x))\beta(y) = \tau(\beta(x))\alpha(y), \]
	\[ \tau \circ \alpha = \tau \text{ and } \tau \circ \beta = \tau, \]
	are satisfied then $(A,[\cdot,\ldots,\cdot]_\tau,\alpha,\beta)$ is an $(n+1)$-BiHom-Lie algebra. We say that this algebra is induced by $(A,[\cdot,\ldots,\cdot],\alpha,\beta)$.
\end{theorem}

We first focus on the ternary case.
For a given BiHom-Lie algebra $(\mathfrak{g},[\cdot,\cdot])$ and a $[\cdot,\cdot]$-trace map $\tau: \mathfrak{g} \rightarrow \mathbb{K}$,  the ternary induced bracket is then given by
\begin{eqnarray}\label{tau}
	[x,y,z]_\tau:=\tau(x)[y,z] -\tau(y)[x,z] +\tau(z)[x,y].
\end{eqnarray}
Now we have the following theorem.

\begin{theorem}\label{AAA}
	Let $ (\mathfrak{g}, [\cdot,\cdot], \alpha, \beta ) $ be a BiHom-Lie algebra.   Let $D: \mathfrak{g} \rightarrow \mathfrak{g}$ be an $(\alpha^s,\beta^r)$-derivation of $ (\mathfrak{g}, [\cdot,\cdot], \alpha, \beta ) $.  If the following identity holds, for all $ x,y,z \in \mathfrak{g},$ 
	\[ \alpha^s\beta^r ( \tau(D(x)) [y,z])-\alpha^s\beta^r ( \tau(D(y)) [x,z])+ \alpha^s\beta^r ( \tau(D(z)) [x,y])= 0,\]
	then\\  $D$ is an $(\alpha^s,\beta^r)$-derivation of the induced ternary BiHom-Lie algebra $ (\mathfrak{g}, [\cdot,\cdot,\cdot]_\tau, \alpha, \beta ) $.
\end{theorem}

\begin{proof}
	In the sequel, for simplicity we drop the $\tau$ from the ternary bracket.
	We have  to prove that 
	\[
	D[x,y,z]=[D(x), \alpha^s\beta^r(y),  \alpha^s\beta^r(z)] +[ \alpha^s\beta^r(x), D(y),  \alpha^s\beta^r(z) ] + [ \alpha^s\beta^r(x),  \alpha^s\beta^r(y), D(z)].
	\]
	By applying $D$ to each side of equation (\ref{tau}), we get

	\begin{eqnarray}
		LHS &=&D[x,y,z]= \tau(x)D[y,z] -\tau(y)D[x,z] +\tau(z)D[x,y] \nonumber \\
		&=&\tau(x)([D(y), \alpha^s\beta^r(z)] + [\alpha^s\beta^r(y), D(z)]) -  \nonumber \\
		&&  \tau(y)([D(x), \alpha^s\beta^r(z)] - [\alpha^s\beta^r(x), D(z)]) +  \nonumber  \\
		&&        \tau(z)([D(x), \alpha^s\beta^r(y)] + [\alpha^s\beta^r(x), D(y)]),\nonumber 
	\end{eqnarray}
	while,
	\begin{align}
		& RHS= [D(x), \alpha^s\beta^r(y),  \alpha^s\beta^r(z)] +[\alpha^s\beta^r(x), D(y),  \alpha^s\beta^r(z) ] + [ \alpha^s\beta^r(x),  \alpha^s\beta^r(y), D(z)] \nonumber \\
		&=\tau(D(x))[\alpha^s\beta^r(y), \alpha^s\beta^r(z)] -  \tau(\alpha^s\beta^r(y))[\alpha^s\beta^r(x), D(z)] +  \tau(\alpha^s\beta^r(z))[D(x),\alpha^s\beta^r(y)] \nonumber \\
		&+\tau(\alpha^s\beta^r(x))[D(y), \alpha^s\beta^r(z)] -  \tau(D(y))[\alpha^s\beta^r(x), \alpha^s\beta^r(z)] +  \tau(\alpha^s\beta^r(z))[\alpha^s\beta^r(x), D(y)] \nonumber \\
		&+\tau(\alpha^s\beta^r(x))[ \alpha^s\beta^r(y), D(z)] -  \tau(\alpha^s\beta^r(y))[D(x), \alpha^s\beta^r(z)] +  \tau(D(z))[\alpha^s\beta^r(x), \alpha^s\beta^r(y)]. \nonumber 
	\end{align}
	Using  the fact that \[ \tau \circ \alpha = \tau \text{ and } \tau \circ \beta = \tau, \]
	we can rewrite the right hand side as
	\begin{eqnarray}
		RHS&=&\tau(D(x))[\alpha^s\beta^r(y), \alpha^s\beta^r(z)] -  \tau(y)[\alpha^s\beta^r(x), D(z)] +  \tau(z)[D(x),\alpha^s\beta^r(y)] \nonumber \\
		&+&\tau(x)[D(y), \alpha^s\beta^r(z)] -  \tau(D(y))[\alpha^s\beta^r(x), \alpha^s\beta^r(z)] +  \tau(z)[\alpha^s\beta^r(x), D(y)] \nonumber \\
		&+&\tau(x)[ \alpha^s\beta^r(y), D(z)] -  \tau(y)[D(x), \alpha^s\beta^r(z)] +  \tau(D(z))[\alpha^s\beta^r(x), \alpha^s\beta^r(y)], \nonumber 
	\end{eqnarray}
	Thus the difference between the right hand side and the left hand side is given by
	\begin{eqnarray}
		RHS-LHS&=&\tau(D(x))[\alpha^s\beta^r(y), \alpha^s\beta^r(z)]  -  \tau(D(y))[\alpha^s\beta^r(x), \alpha^s\beta^r(z)] + \nonumber \\
		&& \tau(D(z))[\alpha^s\beta^r(x), \alpha^s\beta^r(y)], \nonumber \\
		&=& \tau(D(x))\alpha^s\beta^r[y,z] -  \tau(D(y))\alpha^s\beta^r[x,z] +  \tau(D(z))\alpha^s\beta^r[x,y]. \nonumber
	\end{eqnarray}
	We then obtain the result by assuming that the following identity holds, $\forall x,y,z \in \mathfrak{g}, $
	\[ \alpha^s\beta^r (\tau(D(x))  [y,z] -  \tau(D(y)) [x,z] +  \tau(D(z)) [x,y]) = 0.\]
	This ends the proof.
\end{proof}

Similar computations lead to a generalization of  Theorem~\ref{AAA}  to $n$-ary case.

\begin{theorem}
	Let $ (\mathfrak{g}, [\cdot,\ldots,\cdot], \alpha, \beta ) $ be an $n$-BiHom-Lie algebra.   Let $D: \mathfrak{g} \rightarrow \mathfrak{g}$ be an $(\alpha^s,\beta^r)$-derivation of $ (\mathfrak{g}, [\cdot,\ldots, \cdot], \alpha, \beta ) $.  If the following identity holds 
	\[\forall x_1,\ldots, x_n \in \mathfrak{g}, \sum_{i=1}^{n}(-1)^{i-1} \alpha^s\beta^r( \tau(D(x_i))  [x_1,\ldots,\hat{x_i}\ldots,x_n]) = 0,\]
	
	then $D$ is a derivation of the induced  $(n+1)$-BiHom-Lie algebra $ (\mathfrak{g}, [\cdot,\ldots, \cdot,\cdot], \alpha, \beta ) $.
\end{theorem}

\begin{example}
	We consider the $2$-dimensional BiHom-Lie algebra $\mathfrak{g}$ with a basis $\{e_1, e_2\}$, where the map  $\alpha$ is given by $\alpha(e_1)=e_1$ and $\alpha(e_2)=\frac{1}{m}e_1+\frac{n-1}{n}e_2$ and $\beta$ is the identity map (see  \cite{Sheng-Qi}).  The bracket is given by $[e_1,e_1]=0,\;  [e_1,e_2]=me_2-ne_1, \; [e_2,e_1]=(n-1)e_1-\frac{m(n-1)}{n}e_2$ and $[e_2,e_2]=-\frac{n}{m}e_1+e_2$, where  $m$ and $n$ are scalars such that $m,n \neq0$. \\  A direct computation gives that
	\[
	\alpha^s(e_1)=e_1, \quad    \alpha^s(e_2)=\frac{n^s-(n-1)^s}{ n^{(s-1)} m}e_1+ \frac{(n-1)^s}{n^s}e_2.
	\]
	We set  $D(e_1)=ae_1+be_2$ and $D(e_2)=ce_2+de_2$. 
	Finding the conditions on the parameters $a, b, c, d$ such that $D$ is an $(\alpha^s,\beta^r)$-derivation and solving the system, we obtain the following $(\alpha^s,\beta^r)$-derivations:
	
	$$\begin{array}{llll}
	\text{For} \; n=1, &  D(e_1)=ae_1; & D(e_2)=be_1+(a-mb)e_2. \\
	\text{For} \; n \neq 1, & D(e_1)=0; & D(e_2)=a(e_1+ \frac{m}{n}e_2).
	\end{array}$$

	Now, seeking  for  linear form $\tau$ that are $[\cdot,\cdot]$-trace, we obtain first 
	\[
	n \tau(e_1)=m \tau(e_2).
	\]
	The condition $\tau(\alpha(x))\beta(y) = \tau(\beta(x))\alpha(y) $ implies that $\tau(e_1)=0=\tau(e_2)$ and thus the form $\tau$ is identically trivial. Therefore the ternary bracket is trivial and any linear map  is a derivation.

\end{example}


\begin{thebibliography}{99.}%
% Use the following syntax and markup for your references if 
% the subject of your book is from the field 
% "Mathematics, Physics, Statistics, Computer Science"
%
% Contribution 
\bibitem{ams:MabroukRep}  
{\it Ammar F.,  Mabrouk S.,  Makhlouf A.,}  
Representations and cohomology of $n$-ary multiplicative Hom-Nambu-Lie algebras, J. Geom.  Physics, 61 (2011), 10, 1898--1913..

\bibitem{AMS} {\it Ataguema H.,  Makhlouf A., Silvestrov S.,} 
Generalization of $n$-ary Nambu algebras and beyond, 
Journal of  Mathematical  Physics,  50 (2009),   8, 083501.


\bibitem{n-LieAMS}
{\it Arnlind J., Makhlouf A., Silvestrov S.,} 
Construction of $n$-Lie algebras and $n$-ary Hom-Nambu-Lie algebras, 
Journal of Mathematical Physics, 52 (2011), 4, 123502, 13 pp. 




\bibitem{bkp}
{\it Beites P.,  Kaygorodov I., Popov Yu.,} 
Generalized derivations of multiplicative $n$-ary $Hom$-$\Omega$ color algebras,
Bulletin of the Malaysian Mathematical Sciences Society, 42 (2019), 1, 315--335.


\bibitem{BEM} {\it Ben Abdeljelil A., Elhamdadi M., Makhlouf A.,} Derivations of ternary Lie algebras and generalizations, Int. Electron. J. Algebra, 21 (2017), 55--75. 


\bibitem{CG} {\it   Caenepeel, S. , Goyvaerts, I. }
Monoidal Hom-Hopf algebras,
Comm. Algebra, 39 (2011), 6, 2216--2240.






\bibitem{Filippov:nLie}
{\it Filippov V.,}
$n$-Lie algebras,
Siberian Mathematical Journal, 26 (1985), 6, 126--140.

\bibitem{Bihom1}
{\it Graziani G., Makhlouf A.,  Menini. C.,  Panaite F.,}
BiHom-associative algebras,
BiHom-Lie algebras and BiHom-bialgebras,
Symmetry, Integrability and geometry,  SIGMA 11 (2015), 086, 34 pages.


\bibitem{Jack} {\it Jacobson N.}
Lie and Jordan triple systems,
Amer. J. Math. 71, (1949). 149-170. 

\bibitem{hom06}
{\it Hartwig J., Larsson D., Silvestrov S.,} 
Deformations of Lie algebras using $\sigma$-derivations, 
Journal of Algebra, 295 (2006), 2, 314--361. 


\bibitem{kay1}  {\it Kaygorodov I.,} $\delta$-Derivations of simple finite-dimensional Jordan superalgebras,
Algebra and Logic, 46 (2007), 5, 318--329.


\bibitem{kay_lie} {\it Kaygorodov I.,}
$\delta$-Derivations of classical Lie superalgebras,
Siberian Mathematical Journal, 50 (2009), 3, 434-449.


\bibitem{kay_lie2}  {\it Kaygorodov I.,}
$\delta$-Superderivations of simple finite-dimensional Jordan and Lie superalgebras,
Algebra and Logic,  49 (2010), 2, 130--144.



\bibitem{kay12mz}   {\it Kaygorodov I.,}
$\delta$-Superderivations of semisimple finite-dimensional Jordan  superalgebras,
Mathematical Notes, 91 (2012), 2, 187--197.

\bibitem{kay12izv}   {\it Kaygorodov I.,}
$\delta$-Derivations of $n$-ary algebras,
Izvestiya: Mathematics, 76 (2012), 5, 1150--1162.




\bibitem{kay11aa} {\it Kaygorodov I.,}
$(n+1)$-Ary derivations of simple $n$-ary algebras, 
Algebra and Logic, 50 (2011), 5, 470--471.

\bibitem{kay14sp} {\it Kaygorodov I.,}
$(n+1)$-Ary derivations of simple Malcev algebras,
St. Peterburg Math. Journal, 23 (2014), 4, 575--585.


\bibitem{kay14mz}
{\it Kaygorodov I.,}
$(n+1)$-Ary derivations of semisimple Filippov algebras,
Mathematical Notes, 96 (2014), 2, 208--216.

\bibitem{Kokh}
{\it Kaygorodov I., Okhapkina E.,}
$\delta$-Derivations of semisimple finite-dimensional structurable algebras, 
Journal of Algebra and its Applications, 13 (2014), 4, 1350130, 12 pp.

\bibitem{KP}
{\it Kaygorodov I., Popov Yu.,}
A characterization of nilpotent nonassociative algebras by invertible Leibniz-derivations,
Journal of Algebra, 456 (2016), 323--347.


\bibitem{KP16}
{\it Kaygorodov I., Popov Yu.,}
Generalized Derivations of (color) $n$-ary algebras, 
Linear Multilinear Algebra, 64 (2016), 6, 1086--1106.



\bibitem{KP16com}
{\it Kaygorodov I., Popov Yu.,}
Commentary  to:  Generalized derivations of Lie triple systems, 
Open Mathematics, 14 (2016), 543--544.

\bibitem{KMS} {\it Kitouni A.,  Makhlouf A.,  Silvestrov S.,}  On $n$-ary Generalization of BiHom-Lie algebras and BiHom-Associative Algebras.,  arXiv:1812.00094,  2018.

\bibitem{KN03} {\it Komatsu H., Nakajima A.,}
Generalized derivations of associative algebras, 
Quaestiones Mathematicae, 26 (2003), 2, 213--235.



\bibitem{LL00}
{\it Leger G., Luks E.,}
Generalized derivations of Lie algebras,
Journal of Algebra, 228 (2000), 1, 165--203.

\bibitem{MS} 
{\it Makhlouf A., Silvestrov A.,} 
Hom-algebra structures, Journal of
Generalized Lie Theory and Applications, 2 (2008), 2, 51--64.


\bibitem{Sheng}
{\it Sheng Y.} 
Representations of Hom-Lie algebras,
Algebr. Represent. Theory, 15 (2012), 6, 1081--1098.

\bibitem{Sheng-Qi}
{\it Sheng Y. , Qi H.} 
Representations of BiHom-Lie algebras,
arXiv: 1610.04302,  (2016).





\bibitem{takhtajan} 
{\it Takhtajan L.,} 
On foundation of the generalized Nambu mechanics, Comm. Math. Phys., 160 (1994), 295--315.


\bibitem{takhtajan2} {\it Takhtajan L.,} 
Leibniz and Lie algebra structures for Nambu algebra, 
Lett. Math. Phys., 39 (1997), 127--141.





\bibitem{ZZ10} {\it Zhang R., Zhang Y., }
Generalized derivations of Lie superalgebras, 
Communications in Algebra, 38 (2010), 10, 3737--3751.



\bibitem{zhel}
{\it Zhelyabin V., Kaygorodov I.,} 
On $\delta$-superderivations of simple superalgebras of Jordan brackets,
St. Petersburg Mathematical Journal, 23 (2012), 4, 665--677.



\bibitem{ZCM16} {\it Zhou J., Chen L., Ma Y.,}
Generalized Derivations of Hom-Lie triple systems, 
Bulletin of the Malaysian Mathematical Sciences Society, 41 (2018), 2, 637--656.




\bibitem{zus10} {\it Zusmanovich~P.,}  On $\delta$-derivations of Lie algebras and superalgebras,
Journal of Algebra, 324 (2010), 12, 3470--3486.


\end{thebibliography}
\end{document}